\documentclass[11pt,a4]{article}

\usepackage{latexsym}
\usepackage[dvips]{graphicx}
\usepackage{amsmath}
\usepackage{amssymb}

\usepackage{times}
\usepackage{bm}
\usepackage{natbib}

\usepackage{enumerate}
\usepackage{amsfonts}
\usepackage{amsthm}
\usepackage{mathrsfs}
\usepackage{ascmac}

\newcommand{\half}{\frac{1}{2} }

\newcommand{\leftb}{\left( }
\newcommand{\rightb}{\right) }
\renewcommand{\cal}[1]{\mathcal{#1}}
\newcommand{\real}{\mathbf{R}}

\theoremstyle{definition}

\newtheorem{theorem}{Theorem}
\newtheorem{corollary}[theorem]{Corollary}

\newtheorem{lemma}{Lemma}
\newtheorem{remark}{Remark}
\newtheorem{definition}{Definition}

\begin{document}

\setlength{\baselineskip}{20pt}

\title{An analytic example of latent information prior}

\author{F. TANAKA}


\maketitle

\begin{abstract}
Recently Komaki proposed latent information priors as an objective prior.
In this short article, we consider the one-step ahead prediction based on one-sample under the binomial model.
In this specific case, the latent information prior is derived analytically and shown to be a discrete prior.
It verifies the numerical result by Komaki.
As a by-product, we obtain the minimal complete class, the minimax predictive distribution.
\end{abstract}


\section{Introduction}

Recently \citet{Komaki2011} proposed latent information priors as the extension of reference priors by~\citet{Bernardo1979}, 
In finite sample, finding these priors are extremely hard unless we use numerical methods.
In this short article, we consider the most simple case and give the explicit form of the latent information prior analytically.
The process to derive it would help us understand properties of latent information priors.

Suppose that the observation $x$ comes from the Bernoulli model $p_{\theta}(x) = \theta^{x} (1-\theta )^{1-x}$, $x=0,1$ and  $0 \leq \theta \leq 1$.
The future outcome $y$ is also generated under the same model.
We consider constructing the one-step ahead predictive distribution based on the observation.
In this specific case, any predictive distribution is determined by an estimate of $\theta $, which is denoted by $\delta (x)$.
The explicit form is $p_{\delta}(y) =  (\delta (x))^{y} (1-\delta (x))^{1-y}$.
From now on, we argue decision-theoretic concepts. Thus, we often call $\delta $ or $p_{\delta}$ a decision instead of an estimate or a predictive distribution.
For basic facts and terminologies of decision theory, readers may consult e.g., \citet{Ferguson1967}.

We adopt the Kullback-Leibler divergence from
 the true probability distribution $p_{\theta}$ to the predictive distribution $p_{\delta}$ as a loss function.
It is simply written by
\begin{align*}
L(\theta, \delta (x)) =  D(p_{\theta} || p_{\delta})= \theta \log \frac{\theta}{ \delta (x)}
 + (1-\theta) \log \frac{1-\theta }{ 1-\delta (x)}.
\end{align*}
According to usual manner, we define $0 \log 0 = 0$ and $c \log 0 = -\infty, (c > 0)$. 
Then, we obtain the risk function by taking the average of the loss function with respect to $x$, 
\begin{align*}
R_{\delta}(\theta ) &= -S(\theta ) + \theta^{2} \log \frac{1}{ \delta (1)}
  + \theta (1 -\theta ) \log  \frac{1}{ 1-\delta (1)}\\
& \qquad  {}+ \theta (1 -\theta ) \log  \frac{1}{ \delta (0)}
    + (1 -\theta )^{2} \log  \frac{1}{ 1-\delta (0)},   
\end{align*}
where $S (\theta ) = -\theta \log \theta - (1-\theta ) \log (1-\theta )$ is the binary entropy.
Compared to the same problem with the squared loss, there is no equalizer because of the nonlinearity of $S(\theta )$.
Thus, even in this very simple problem, it seems difficult to find the minimax decision without any resort
 to numerical methods.
However, exceptionally we are able to find the minimax decision and the latent information prior.

\section{Result}

\subsection{Minimal complete class}

For simplicity, we consider only nonrandomized decision functions.
Each nonrandomized decision function is determined by $\delta(0)$ and $\delta (1)$ and identified with a point in the square $[0,1] \times [0,1]$.
From now on, we write $\delta_{0}, \delta_{1} $ rather than $\delta(0), \delta (1)$.
Let us denote the whole nonrandomized decisions as $\cal{D} = \{  (\delta_{0}, \delta_{1})\in \real^{2}:\ 0 \leq \delta_{0} \leq 1,\ 0 \leq \delta_{1} \leq 1 \}.$
It would be helpful to decompose $\cal{D}$ into some disjoint classes: 
\begin{align*}
 & \cal{C}_{<} = \{ (\delta_{0}, \delta_{1})\in \real^{2}: \ 0 < \delta_{0} < \delta_{1} < 1 \}, 
\quad \cal{C}_{=} = \{ (\delta_{0}, \delta_{1})\in \real^{2}: \ 0 < \delta_{0} = \delta_{1} < 1 \}, \\
 & \cal{C}_{>} = \{ (\delta_{0}, \delta_{1})\in \real^{2}: \ 0 < \delta_{1} < \delta_{0} < 1 \}, 
\quad  \partial \cal{D} =  \{ (\delta_{0}, \delta_{1})\in \real^{2}: \ \delta_{0}=0,1 \text{ or } \delta_{1}=0,1 \}.
\end{align*}
For later convenience, we also define the decision based on the maximum likelihood, $\delta_{MLE}$ and
 related class of decisions.  
$ \cal{M} = \{ \delta_{MLE} \} = \{ (\delta_{0}, \delta_{1}) =(0, 1) \}, \quad \cal{P} = \partial \cal{D} \setminus \cal{M}.$
Clearly we see that $\delta_{MLE}$ is Bayes with respect to a two-point distribution
$\pi_{M} (\{ 0 \} ) = \xi$, $\pi_{M} (\{ 1 \} ) = 1-\xi$, $0 < \xi < 1$.
According to Hartigan's terminology, $\pi_{M}$ is the maximum likelihood prior (See, \citet{Hartigan1998}).


It is easy to see that the whole class of Bayesian predictive distributions is $\cal{C}_{<} \cup \cal{C}_{=} \cup \cal{M}$.
According to \citet{Komaki2011}, it is essentially complete class.
However, this fact is not enough to find the minimax decision.
We show a slightly strong result of completeness in our specific example. 
\begin{theorem}
\label{minicomp}
The class $\cal{A} = \cal{C}_{<} \cup \cal{C}_{=} \cup \cal{M}$ is the minimal complete class (and therefore it constitutes of all admissible decisions). 
\end{theorem}

Key lemma is as follows:
\begin{lemma}
\label{keyineq}
When $0 < \delta_{1} < \delta_{0} < 1$, there exists $\mu > 0$ satisfying
the following inequality:
\begin{align}
\left\{ \log \frac{ 1- \delta_{1}}{   (\delta_{0} + 1 -\delta_{1}) (1 -\delta_{0}) }   \right\}
  \left\{ \log \frac{ \delta_{0}}{   (\delta_{0} + 1 -\delta_{1})  \delta_{1} }   \right\}
  > \mu^{2} > \{  \log ( \delta_{0} +1 -\delta_{1}) \}^{2}  \label{ineq:geods}
\end{align}
\end{lemma}
Although the proof of lemma is quite tedious, it is straightforward. We omit the proof.

\begin{proof}[proof of Theorem~\ref{minicomp}]

First, we show the admissibility of each class.
By definition, it is trivial to see that every $\delta \in \cal{C}_{=}$ is admissible.
The maximum likelihood estimate $\delta_{MLE}$ is also admissible under the Kullback-Leibler loss.
As we shall see later, $\delta \in \cal{C}_{<}$ is written as the unique Bayes decision with respect to a prior distribution.
Thus, $\cal{C}_{<}$ is admissible.

Next we show the following statements, which imply the completeness of the class $\cal{A}$.
(i) For every $\delta \in \cal{P}$, the maximum likelihood estimate, $\delta_{MLE} $, dominates $\delta $.
(ii) For every $\delta \in \cal{C}_{>}$, there exists $\delta' \in \cal{C}_{<}$ dominating $\delta $.

Trivially (i) holds.
We show (ii), which requires a bit more calculation.
We consider the risk difference,
\begin{align*}
\Delta R(\theta) & = R_{\delta }(\theta ) - R_{\delta'}( \theta)  
 =  F \theta^{2} - G \theta + H,
\end{align*}
where
\begin{align*}
F = \log \leftb  \frac{ \delta'_{1}}{ \delta_{1} } \frac{ 1-\delta_{1}}{ 1-\delta'_{1} }  \frac{ \delta_{0}}{ \delta'_{0} } 
  \frac{ 1-\delta'_{0}}{ 1-\delta_{0} }  \rightb, \ G= \log \leftb \frac{ 1-\delta'_{0} }{ 1-\delta_{0} } \frac{ 1-\delta_{1}}{ 1-\delta'_{1} }  
        \frac{ 1-\delta'_{0}}{ 1-\delta_{0} }  \frac{ \delta_{0}}{ \delta'_{0} }  \rightb, \ H =  \log \frac{1-\delta'_{0}}{1-\delta_{0}}. 
\end{align*}

Assuming $0 < \delta_{1} < \delta_{0} < 0$, we show that there exists $\delta' \in \cal{C}_{<}$
  such that $\Delta R (\theta) > 0$ for any $\theta$.
When $0 < \delta'_{0} < \delta'_{1} < 0$, it is straightforward to show that $F>0$.
(The inequality $0 < G/(2F) < 1$ also holds, but we need no more it.) 
Thus, we obtain $ \Delta R(\theta) = F\{ \theta - G/(2F) \}^{2} + (FH - G^{2}/4)/F$.
It is enough to show $ FH - G^{2}/4 > 0$.

Now we use Lemma~\ref{keyineq}.
Let us choose a positive parameter $\mu $ satisfying the inequality \eqref{ineq:geods}
 and define $\delta'_{0}= e^{-\mu } \delta_{0}$ and $\delta_{1}' = e^{-\mu } \delta_{1} + (1- e^{-\mu })$. 
Then clearly $\delta' \in \cal{C}_{<}$ and 
\begin{align*}
FH-\frac{G^{2}}{4} 
 &=\left\{ \log  \frac{ 1- (1-\delta_{1}) e^{-\mu} }{ \delta_{1} } \right\}
  \left\{ \log \frac{ 1- \delta_{0} e^{-\mu}  }{ 1-\delta_{0}} \right\} -  \mu^{2}  \\
  & > \left\{ \log  \frac{ \delta_{0}}{  ( \delta_{0} + 1 -\delta_{1}) \delta_{1}} \right\}
  \left\{ \log  \frac{  1-\delta_{1}}{ (  \delta_{0} + 1 -\delta_{1}) (1-\delta_{0}) }  \right\} - \mu^{2} \\
  & > 0.
\end{align*}
Thus, (ii) is proved. 

From the above statements we proved, the minimal complete class is included by $\cal{D} \setminus (\cal{C}_{>} \cup \cal{P}) = \cal{A}$. 
On the other hand, the minimal complete class necessarily includes all admissible decisions $\cal{A}$.
Thus, $\cal{A}$ is the minimal complete class.
\end{proof}

\subsection{Minimax predictive distribution}

Next we find the minimax decision.
From the above argument, it is enough to find the minimax decision among $\cal{C}_{<} \cup  \cal{C}_{=}$.
(For $\delta \in \cal{C}_{>}$, the worst case risk of $\delta $ is strictly larger than that of $\delta' \in \cal{C}_{<}$. 
For $\delta \in \partial \cal{D}$, the worst case risk is $+\infty$.) \\

From straightforward calculation, we obtain
\begin{align}
 \inf_{\delta \in \cal{C}_{=}} \sup_{0 \leq \theta \leq 1} R_{\delta }(\theta )
  \geq \log 2. \label{eq:worstrisk}
\end{align}
As we shall see later, $\delta \in \cal{C}_{<}$ has smaller value of the worst case risk.
Thus, we consider the class $\delta \in \cal{C}_{<}$ to find out the minimax decision.

Recall that the Bayes decision with respect to the beta prior, $ B(a, b)^{-1} \theta^{a-1} (1-\theta )^{b-1}$, $(a>0, b>0)$, 
 is $\delta_{0} = a/(a + b+ 1)$ and $\delta_{1} = (a+1)/(a+b+1)$.
See, e.g., \citet{Aitchison} for the Bayes decision under the Kullback-Leibler loss.
From the above relation, there exists one-to-one correspondence between the hyperparameter $a$ and $b$ and the decision $\delta \in \cal{C}_{<}$.
From now on, $(a, b)$ is identified with $\delta \in \cal{C}_{<}$.

Next, we show that it is enough to consider the minimax decision when $a=b$.
For an arbitrary $\delta \in \cal{C}_{<} $ with distinct $a$ and $b$, 
let $ \tilde{\delta} $ denote as the decision where $a$ and $b$ are swapped in $\delta$.
Then the symmetrized decision, that is, $ \delta_{s} = \delta/2 + \tilde{\delta}/2$
satisfies the following inequality due to the convexity of the Kullback-Leibler divergence,
\begin{align*}
\half R_{\delta} (\theta ) + \half R_{\tilde{\delta}} (\theta ) & > R_{\delta_{s} } ( \theta ). 
\end{align*}
Since $R_{\tilde{\delta}} (\theta) = R_{\delta} (1-\theta)$, we obtain
\begin{align*}
 \sup_{0 \leq \theta \leq 1} R_{\delta}(\theta )
 > \sup_{0 \leq \theta \leq 1} R_{\delta_{s} } ( \theta ). 
\end{align*}
Therefore we may assume $a=b$ without loss of generality.
Then, the risk function is written as
\begin{align*}
R_{\delta} (\theta) &= -S(\theta ) -2 \leftb \log \frac{a+1}{ a} \rightb \theta^{2}
  -2 \leftb  \frac{a+1}{a} \rightb \theta +\log \leftb \frac{2a +1}{ a+1}\rightb.
\end{align*}
From elementary calculation we obtain
\begin{align*}
 \sup_{0 \leq \theta \leq 1} R_{\delta} (\theta)
 = \begin{cases} 
   \log \leftb \frac{2a +1 }{ a+1} \rightb, \ a \geq 1/3, \\
   \frac{1}{2} \log \leftb \frac{a+1}{4a} \rightb + \log \leftb \frac{2a +1 }{ a+1} \rightb, \ 0 < a \leq 1/3. 
 \end{cases} 
\end{align*}
When $a \geq 1/3$, the above quantity is necessarily smaller than $\log 2$.
From the inequality \eqref{eq:worstrisk}, we see that $\delta \in \cal{C}_{=}$ must not be minimax.
It is easily seen that $a=1/3$ corresponds to the unique minimax decision. 
Clearly it has smaller worst case risk than that of the Jeffreys prior ($a=1/2$) and that of the uniform prior ($a=1$).

Let us summarize the above result as Theorem.
\begin{theorem}
\label{mainresult}
In the above setting, the minimax decision is $ \delta_{0} = 1/5, \delta_{1}=4/5 $ and unique.
\end{theorem}

The following concept is  useful.
\begin{definition}
A prior $\pi $ is called a minimax prior if the Bayes decision with respect to $\pi$ is minimax.
\end{definition}

We see that the Beta prior with $a=1/3$, that is, $B(1/3, 1/3)^{-1}\theta^{-2/3} (1-\theta )^{-2/3}$
 is a minimax prior. 
However, we have many minimax priors in this setting.
Indeed, we easily find a necessary and sufficient condition where a prior is a minimax prior.
\begin{corollary}
\label{cond:moment}
In the above setting, a prior distribution $\pi$ is a minimax prior 
 if and only if it satisfies the moment condition, $E(\theta) =1/2$ and $E(\theta^{2}) = 2/5$.
\end{corollary}

\subsection{Latent information prior}

Finally we obtain the latent information prior, which by definition maximizes the conditional mutual information.
As shown in \citet{Komaki2011}, Bayesian predictive distribution based on the latent information prior is minimax
 and thus the latent information prior is among the class of minimax priors.
We only consider the maximization of the conditional mutual information in a class of all minimax priors.

\begin{theorem}
\label{latent}
In the above setting, the latent information prior is $\pi_{LI}(\{ 0 \}) = \pi_{LI} ( \{ 1 \}) = 3/10,\ \pi_{LI}(\{ 1/2 \}) = 2/5$.
The maximum of the conditional mutual information is $\log 5$.
\end{theorem}
Our analytical result exactly coincides with numerically obtained one (See Fig.~1 in \citet{Komaki2011}).
As far as the author knows, it is the only analytical example of the latent information prior on a continuous parameter space.

The essential part of its proof is to show that the support of the latent information prior is equal to or less than four.
We use the following lemma.
\begin{lemma}
\label{finitesupport}
In the above setting, the latent information prior must be a discrete distribution giving mass one to at most four points.
\end{lemma}

\begin{proof}
From Corollary~\ref{cond:moment}, the conditional mutual information $I(\pi)$ has the simple form,
\begin{equation}
\label{eq:mutual} I(\pi) = - \int_{[0,1]} S(\theta) \pi (d \theta) + \frac{2}{5} \log 2 + \log 5
\end{equation}
for every minimax prior $\pi$. Constant terms are due to the moment condition.

Suppose that the latent information prior $\pi_{LI}$ has five disjoint closed intervals $I_{1},\dots, I_{5}$
 of positive probabilities. We show contradiction below.
We define positive finite measures $r_{1}, \dots, r_{5}$ by restricting the probability measure to each interval.
Then, we take five real numbers $a_{1}, \dots, a_{5}$ satisfying
\begin{align*}
 \int\! 1 \sum_{j=1}^{5} a_{j} r_{j} (d \theta ) =
 \int\! \theta \sum_{j=1}^{5} a_{j} r_{j} (d \theta ) =
 \int\! \theta^{2} \sum_{j=1}^{5} a_{j} r_{j} (d \theta ) =0, \ \int S(\theta ) \sum_{j=1}^{5} a_{j} r_{j} (d \theta ) = 1.
\end{align*}
The above linear equations might be degenerate if the probability measure has continuous components.
In this case, we replace one interval with a smaller one.

Next, we define a probability measure close to $\pi_{LI}$ as 
\begin{align*}
\tilde{\pi} = \pi_{LI} - \sum_{j=1}^{5} \frac{a_{j} r_{j}}{ |a_{1}| + \cdots + |a_{5}| }.
\end{align*}
Then $\tilde{\pi}$ is a minimax prior and $I(\tilde{\pi}) =  I(\pi_{LI}) + 1/(|a_{1}| + \cdots + |a_{5}| ) > I(\pi_{LI}) $, which yields contradiction. 
\end{proof}

\begin{remark}
At most, five points are necessary.
For example, $S(\theta ) + 3 \theta^{2} - 3 \theta - 0.04$ has four distinct zeros.
If $\pi_{LI}$ has the nonzero weight to each zero, then four equations are no longer independent. 
\end{remark}

\begin{proof}[proof of Theorem~\ref{latent}]
From Lemma~\ref{finitesupport}, the latent information prior has positive weights only on four points.
Since the maximization problem is symmetric about $\theta =1/2$, we may assume that four points are $x,y, 1-y$ and $1-x$, $(0 \leq x \leq y \leq 1/2)$.
The latent information prior is parametrized as $\pi(\{ x \}) = \pi ( \{ 1-x \}) = \alpha$ and $\pi(\{ y \}) = \pi ( \{ 1-y \}) = \beta$.
We consider minimizing the function
\begin{align*}
\int_{[0,1]} S(\theta ) \pi (d \theta ) = 2 S(x) \alpha + 2 S(y) \beta  
\end{align*}
under the moment condition in Corollary~\ref{cond:moment}, the normalization $2\alpha + 2 \beta =1$, and the positivity $\alpha \geq0, \beta \geq 0$.
We easily obtain the minimum at $x=0, y=1/2, \alpha = 3/10$ and $\beta = 2/10$.
\end{proof}

The maximization of the conditional mutual information assures the uniqueness of the resulting objective prior.
However, as our analytical example indicates, it may yield a discrete prior.
In our example, the resulting discrete prior is not a result of discretized numerical maximization.
At least, the latent information prior in this specific example is not desirable as a general-purpose objective prior, e.g., it is not used to construct a credible region.

\section*{Acknowledgement}
This work was supported by the Grant-in-Aid for Young Scientists (B) (No. 24700273)
 and the Grant-in-Aid for Scientific Research (B) (No. 26280005).


\end{document}